\documentclass[11pt,letterpaper]{amsart}

\usepackage{hyperref}
\usepackage{cleveref}

\usepackage[margin=1in]{geometry}

\usepackage{amssymb}
\usepackage{amsmath}
\usepackage{amsthm}
\usepackage{wasysym}
\usepackage{enumitem}
\usepackage{mathtools}
\usepackage{pxfonts}
\usepackage{dsfont}
\usepackage{stmaryrd}
\usepackage[english]{babel}
\usepackage[T1]{fontenc}
\usepackage[utf8]{inputenc}

\usepackage{tikz-cd}
\usepackage{bbold}
\usepackage{verbatim}
\usetikzlibrary{arrows.meta}
\usetikzlibrary {decorations.pathmorphing}

\newtheorem{thm}{Theorem}[section]
\newtheorem{lemma}[thm]{Lemma}
\newtheorem{cor}[thm]{Corollary}
\newtheorem{prop}[thm]{Proposition}
\newtheorem*{innerrep}{\repname}

\newenvironment{thmrep}[1]
  {\def\repname{Theorem \ref{#1}}\begin{innerrep}}
  {\end{innerrep}}

\newenvironment{correp}[1]
  {\def\repname{Corollary \ref{#1}}\begin{innerrep}}
  {\end{innerrep}}
\theoremstyle{definition}
\newtheorem{definition}[thm]{Definition}

\theoremstyle{remark}
\newtheorem{remark}[thm]{Remark}

\newtheorem{claim}[thm]{Claim}
\newtheorem{question}[thm]{Question}
\newtheorem{obs}[thm]{Observation}

\numberwithin{equation}{section}

\newcommand{\RR}{\mathbf{R}}

\newcommand{\acts}{\curvearrowright}
\newcommand{\G}{\mathcal{G}}

\newcommand{\Z}{\mathbf{Z}}
\newcommand{\N}{\mathbf{N}}

\DeclareMathOperator{\Cay}{\mathrm{Cay}}
\DeclareMathOperator{\Sch}{\mathrm{Sch}}

\DeclareMathOperator{\id}{id}
\DeclareMathOperator{\cost}{cost}

\DeclareMathOperator{\Leb}{Leb}

\usepackage[draft,inline,nomargin,index]{fixme}

\fxsetup{theme=color,mode=multiuser}
\FXRegisterAuthor{kw}{akw}{\color{red}KW}
\FXRegisterAuthor{ap}{aap}{\color{purple}AP}
\definecolor{DarkGreen}{RGB}{30,150,100}
\FXRegisterAuthor{ct}{act}{\color{DarkGreen}CT}

\title{Cost of one-relator groups}
\author[Antoine Poulin]{Antoine Poulin}
\address[Antoine Poulin]{Department of Mathematics, McGill University, Montreal, Canada}
\email{antoine.poulin@mail.mcgill.ca}
\author[Konrad Wr\'{o}bel]{Konrad Wr\'{o}bel}
\address[Konrad Wr\'{o}bel]{Department of Mathematics, The University of Texas at Austin, USA}
\email{konrad.wrobel@austin.utexas.edu}


\begin{document}
\begin{abstract}
    For any infinite one-relator group $\Gamma=\langle S \mid w^m\rangle$ with $w$ not a proper power and $m\ge 1$, we prove that $\cost(\Gamma)=|S|-\frac{1}{m}$. For such groups, this gives $\beta^{(2)}_1(\Gamma)=\cost(\Gamma) - 1$, answering a special case of Gaboriau's question on the relationship between cost and first $\ell^2$-Betti number.
\end{abstract}

\maketitle

\section{Introduction} 
In \cite{Ga02L2}, Gaboriau proved the fundamental inequality $\cost(\Gamma)-1\ge \beta^{(2)}_1(\Gamma)$ relating the cost of an infinite group $\Gamma$ to its first $\ell^2$-Betti number, and asked:
\begin{question}[Gaboriau]\label{quest:costvsl2}
    Is $\cost(\Gamma) - 1 = \beta^{(2)}_1(\Gamma)$ for every countably infinite group $\Gamma$? 
\end{question}

Question \ref{quest:costvsl2} has a positive answer in many cases, including treeable groups \cite{Ga00} and many classes of groups with vanishing first $\ell^2$-Betti number, such as products of infinite groups \cite{Ga00}, residually finite boundedly generated groups \cite{article:shusterman}, inner amenable groups \cite{TD:InnerAmenable}, and Kazhdan groups \cite{HutchcroftPete}. 

One-relator groups are the basic class of groups obtained from free groups by imposing a single relation. Despite the simplicity of their presentations, their algebraic and geometric behavior is highly varied, and they have served as a central testing ground in combinatorial and geometric group theory for nearly a century (for a recent survey, see \cite{linton2025theoryonerelatorgroupshistory}). 
Dicks and Linnell computed the first $\ell^2$-Betti numbers of one-relator groups \cite{DicksLinnell}, yielding many examples with nonvanishing first $\ell^2$-Betti number. 
One-relator groups thus provide a broad and classical test case for Question \ref{quest:costvsl2} beyond the setting in which $\beta_1^{(2)}(\Gamma) = 0$.

We compute the cost of one-relator groups and answer Question \ref{quest:costvsl2} for this class. 
\begin{thm}\label{thm:mainl2}
    Suppose that $\Gamma = \langle S \mid w^m \rangle$ is an infinite one-relator group with $w$ not a proper power and $m\ge 1$. Then,
    \[\cost(\Gamma) -1= \beta_1^{(2)}(\Gamma)=|S|-1 - \frac{1}{m}.\]
\end{thm}

Our result recovers the Dicks--Linnell formula for the first $\ell^2$-Betti number. Theorem \ref{thm:mainl2} is a special case of a result about one-relator products.
\begin{thmrep}{thm:oneRelatorProduct}
    Let $\Gamma=(\Lambda\ast\mathbf Z )/\langle\!\langle g^m\rangle\!\rangle$ where $\Lambda$ is locally indicable, $g\in \Lambda\ast\Z$, and $m\ge 1$. Assume $g$ is not conjugate to an element of $\Lambda$ and is not a proper power.
    Then 
    \[
    \cost(\Gamma)\le \cost(\Lambda)+1-\frac{1}{m}.
    \]
    Moreover, if $\beta_1^{(2)}(\Lambda)=\cost(\Lambda)-1$, then $\beta_1^{(2)}(\Gamma)=\cost(\Gamma)-1=\cost(\Lambda)-\frac{1}{m}$.
\end{thmrep}

The main technical ingredient in proving the upper bound on cost is a cutting method for Schreier graphs of certain actions (see Theorem \ref{thm:cutRelator}). The method applies when, after passing to a suitable locally indicable quotient, the relevant relation becomes embedded.
As a special case, in a locally indicable group, given an embedded relation $w$ and a generator $s$ appearing in $w$, we may cut a $(1-\varepsilon)$-proportion of edges labeled by $s$ in the Schreier graph.
In the one-relator product setting, the quotient $(\Lambda\ast \Z )/\langle\!\langle g\rangle\!\rangle$ is locally indicable via classical one-relator product theory \cite{Howie:Freiheitssatz, Howie:locInd}, allowing us to make use of the full Theorem \ref{thm:cutRelator} to cut a $\frac{1}{m}$ proportion of edges labeled by $s$ while maintaining connectivity.

The connection to $\ell^2$-Betti numbers in Theorem \ref{thm:oneRelatorProduct} arises from the lower bound on first $\ell^2$-Betti number established by \cite[Theorem 3.2]{PetersonThom} in terms of imposed relations and Gaboriau's inequality $\beta_1^{(2)}(\Gamma)\le\cost(\Gamma)-1$ for infinite groups.

By applying Theorem \ref{thm:oneRelatorProduct} inductively, we can compute cost and first $\ell^2$-Betti number for groups with finite \textit{reducible} presentations without proper powers. 
\begin{correp}{cor:reducible}
    Let $\Gamma=\langle S \mid R\rangle$ be an infinite group with a finite reducible presentation with no proper powers. Then \[\cost(\Gamma)-1=\beta^{(2)}_1(\Gamma)=|S|-|R|-1.\] 
\end{correp}

\subsection{Methods of computing cost}
In his seminal paper \cite{Ga00}, Gaboriau assembled a list of groups whose cost is known --- a \textit{mercuriale}. The computations in the original mercuriale arise from a small number of mechanisms: amenability; the scaling formula under measure equivalence; the additive formula under free products, amalgamated free products, and HNN extensions over amenable groups; approximating $\Gamma$ by an increasing union of subgroups; and criteria involving weak forms of normality or families of commuting subgroups.

The current mercuriale \cite{gaboriauSurvey} contains increasingly subtle applications of the same general mechanisms, together with a new percolation-theoretic method for proving that Kazhdan groups have cost one \cite{HutchcroftPete}. 
None of these methods applies uniformly to one-relator groups. Indeed, techniques that imply that $\cost(\Gamma) = 1$ cannot apply when there are at least three generators. Moreover, while one-relator groups admit iterative decompositions as HNN extensions (known as the \textit{Magnus--Moldavanski\u \i-Masters hierarchy}), the edge groups are free subgroups which need not be amenable. The cost formula for HNN extensions over amenable subgroups therefore does not apply in general.

The method of this paper is of a different form. Rather than deriving the computation from a decomposition of the group, from measure equivalence, or from the presence of suitable subgroups, we work with the presentation and take advantage of a Conradian-order on a locally indicable quotient to give an explicit edge-cutting algorithm. 
To the knowledge of the authors, this is the first use of orderability in a cost argument.

Another question of Gaboriau concerns whether every countably infinite group has \textit{fixed price}, that is, whether all of its free pmp actions have the same cost. 
This question has recently seen renewed activity, with several new positive results \cite{TD:InnerAmenable, FMW, mellick2023gaboriau, Khezeli:Products, bevilacqua2025metric}. 
While our techniques compute the cost of one-relator groups, in Remark \ref{rmk:nofixedprice} we identify an obstruction to using them to prove fixed price.

After this paper appeared, OpenAI produced a group with a reducible presentation that does not have fixed price \cite{OAI:fixedPrice}. This implies Corollary \ref{cor:reducible} is optimal in that an arbitrary free pmp action of a group with finite reducible presentation $\langle S\mid R\rangle$ need not have cost equal to $|S|-|R|-1$.

\subsection*{Acknowledgements}

We would like to thank Miko\l aj Frączyk, Damien Gaboriau, Marco Linton, Sam Mellick, and Anush Tserunyan for helpful discussions and Fran\c cois Le Ma\^itre for various comments. We thank Jinhyuk Son for pointing out a mistake in an earlier draft of this article.

\subsection*{AI Disclosure}

AI was not used to produce text or mathematics in this article. It was used to give comments on preliminary drafts of the paper, specifically in the introduction.

\section{Preliminaries}

\subsection{Cost of graphings, actions and groups}

Let $(X,\mu)$ be a standard probability space. A \textbf{Borel graph} $\mathcal G$  on $X$ is a symmetric irreflexive Borel subset of $X^2$. For $x\in X$, write 
\[
\deg_\G(x)=|\{y\in X: (y,x)\in \G\}|.
\]
The \textbf{cost} of a Borel graph $\G$ is
\[
\mathcal C_\mu(\G)=\frac 12\int_X\deg_\G(x)\;d\mu(x).
\]
For example, for a graph on a finite set, $\mathcal{C}_\mu(\G) = \frac{|E|}{|V|}$ by the handshake lemma.

Let $\Gamma\acts (X,\mu)$ be a free pmp action of a countable group $\Gamma$. A \textbf{graphing} of the action $\Gamma\acts (X,\mu)$ is a Borel graph on $X$ whose connected components are precisely $\Gamma$-orbits on a conull set.

The \textbf{cost} of the action is 
\[
\mathcal C_\mu(\Gamma\acts (X,\mu))=\inf_{\G}\mathcal C_\mu(\G)
\]
where the infimum is taken over all graphings $\G$ of $\Gamma\acts (X,\mu)$.
In general, this infimum need not be attained; a graphing achieves this infimum if and only if it is acyclic \cite{Ga00}. 

Finally, the \textbf{cost of a group} is
\[
\cost(\Gamma)=\inf_{\Gamma \acts (X,\mu)}\mathcal C_\mu(\Gamma\acts(X,\mu))
\]
where the infimum is taken over free pmp actions $\Gamma\acts (X,\mu)$.
For more about cost, see \cite{tioe, gaboriauSurvey}.

A crucial example of a graphing is a Schreier graph. Given an action $\Gamma \acts X$ and a generating set $S$ for $\Gamma$, the \textbf{Schreier graph} is $\Sch(\Gamma \acts X, S) := \left\{(s.x,x)^{\pm1} \in X^2 : x\in X, s\in S\right\}$.

\subsection{Coset processes}
    Let $\Gamma$ be a countable group. For a subgroup $\Delta \le \Gamma$, write $(X_\Delta,\mu_\Delta)=([0,1]^{\Gamma/\Delta},\Leb^{\otimes \Gamma/\Delta})$. Define the \textbf{(Bernoulli) coset process associated to $\Delta$} to be the pmp action $\Gamma \acts (X_\Delta, \mu_\Delta)$ acting via the left shift. If $\mathcal D$ is a countable family of subgroups of $\Gamma$, we call the diagonal action $\Gamma \acts \prod_{\Delta\in \mathcal D} (X_\Delta, \mu_\Delta)$ the \textbf{(Bernoulli) coset process associated to $\mathcal D$}. These are sometimes referred to as generalized Bernoulli shifts in the literature.

We will make use of the following elementary fact about coset processes. 

\begin{prop}\label{prop:finIndexCosetProcess}
    Let $\Delta \le \Lambda\le \Gamma$ with $[\Lambda:\Delta]<\infty$. Suppose $\Gamma \acts (X_\Delta,\mu_\Delta)$ is the coset process associated to $\Delta$. 
    Then there exists a $\Lambda$-equivariant Borel map $\psi: X_0\to \Lambda/\Delta$ defined on a $\Gamma$-invariant conull set $X_0\subseteq X_\Delta$ such that the push-forward measure $\psi_*(\mu_\Delta)$ is uniform on $\Lambda/\Delta$.
\end{prop}
\begin{proof}
    We treat $\Lambda/\Delta$ as a finite subset of $\Gamma/\Delta$. Since the Lebesgue measure on $[0,1]$ is non-atomic, there is a $\Gamma$-invariant conull set $X_0\subseteq X_\Delta$ for which every $x\in X_0$ has a unique minimum among the coordinates labeled by $\Lambda/\Delta$.

    Define now $\psi: X_0\to \Lambda/\Delta$ by letting
    \[
    \psi(x)=\lambda_x\Delta
    \]
    where $\lambda_x\Delta\in \Lambda/\Delta$ is the coordinate at which the minimum is attained.  It is direct to check that the map $\psi$ is measurable and $\Lambda$-equivariant. By $\Lambda$-equivariance, the push forward measure on $\Lambda/\Delta$ is $\Lambda$-invariant and thus uniform. 
\end{proof}

Recall that if $\Gamma \acts (X,\mu)$ and $\Gamma \acts (Y, \nu)$ are two pmp actions of $\Gamma$, we say that a $\Gamma$-equivariant map $\Psi:Y\rightarrow X$ is a \textbf{factor map} if $\Psi_*\nu=\mu$. We call $\Gamma \acts (X,\mu)$ a \textbf{factor} of $\Gamma \acts (Y,\nu)$ and $\Gamma \acts (Y, \nu)$ an \textbf{extension} of $\Gamma\acts (X, \mu)$.

\begin{obs}\label{obs:cosetProcessFactors}
    The conclusion of Proposition \ref{prop:finIndexCosetProcess} holds for any extension of a coset process associated to $\Delta$ by composing $\psi$ with the factor map. Note also that if $\Delta\in\mathcal D$, then the coset process associated to $\Delta$ is a factor of the coset process associated to $\mathcal D$, with the projection being the factor map.
\end{obs}

\subsection{Orders on groups}

In this section, we recall some of the basics of orders on groups. The material can be found in \cite{book:GOD}. Let $\Gamma$ be a group. A \textbf{right-order} on $\Gamma$ is a strict linear order $<$ which is right-invariant, that is for all $g,h,k \in \Gamma, g < h \Leftrightarrow gk < hk.$
We further say that $<$ is a \textbf{Conradian order} if it satisfies Conrad's property: for all $g,h >1$, there exists $n \in \N$ such that $g < g^nh$. An element $g \in \Gamma$ is \textbf{positive} if $g > 1$.

\begin{remark}
    It is more typical to use left-orders in the literature (though classically, people spoke of right-orders, as in \cite{article:Conrad}). For our purposes, it will be convenient to use right-orders as they naturally induce an order relation on an orbit of a left-action $\Gamma \acts X$, through identifying $\Gamma \cong \Gamma.x$. This identification depends on the choice of a point in the orbit, but only up to right multiplication.
\end{remark}

We will say a group is \textbf{right-orderable} if it admits a right-order, and a \textbf{right-ordered} group is a pair $(\Gamma, <)$ where $<$ is a right-order on the group $\Gamma$. We use similar terminology for Conradian orders. A homomorphism $\phi:\Gamma  \rightarrow \Delta$ between ordered groups $(\Gamma,<)$ and $(\Delta, <')$ is \textbf{order-preserving} if whenever $g < h$, we have that $\phi(g) <' \phi(h)$ or $\phi(g)=\phi(h)$.
If $\Lambda \le \Gamma$ is a subgroup of an ordered group $(\Gamma, <)$, then $\Lambda$ is an ordered group with the restricted order. If the order on $\Gamma$ is Conradian, then so is the restricted order on $\Lambda$.

\begin{prop}[{\cite[Corollary 3.2.30]{book:GOD}}]\label{prop:character}
    Suppose that $(\Gamma, <)$ is a finitely generated Conradian-ordered group. Then there is a unique (up to scaling) nontrivial order-preserving homomorphism $\psi:(\Gamma,<) \rightarrow \mathbf R$.
\end{prop}

In the cited reference, the proposition is phrased in terms of convex subgroup jumps. See the discussion at the bottom of page 121 in \cite{book:GOD} for how it applies in the Conradian-ordered case.

\begin{cor}\label{cor:fgsubgroupshavecharacters}
Suppose that $(\Gamma, <)$ is a Conradian-ordered group and that $P\Subset \Gamma$ is a finite set of positive elements. Then, there is a surjective homomorphism $\phi:\langle P\rangle \rightarrow \Z$ with $\phi(P)\ge 0$.
\end{cor}
\begin{proof}
    By Proposition \ref{prop:character}, there is a nontrivial order-preserving homomorphism $\psi': (\langle P\rangle,<)\to \mathbf R$. Since $P$ is finite, there is $k\in \mathbf N$ so that $\psi'(\langle P\rangle)\cong \mathbf Z^k$. By composing with that isomorphism, there is a surjective homomorphism $\psi: \langle P\rangle \to \Z^k$ and $v \in \RR^k$ so that $\langle \psi(g), v\rangle = \psi'(g) \geq 0$ for all $ g\in P$, with equality if and only if $\psi(g) = 0$. 

    Set $\varepsilon:=\min\{ \langle \psi(g), v \rangle\mid g\in P \text{ and } \psi(g)\not=0\}$. Note that the minimum exists and is greater than $0$ since $P$ is finite.  By continuity, there is $w \in \mathbf{Q}^k$ so that $\langle \psi(g), w \rangle \ge \frac{\varepsilon}{2}$ for all $g\in P$ with $\psi(g) \neq 0$. 
    Define a homomorphism $\phi: \langle P\rangle\to \mathbf Q$ by $g \mapsto\langle \psi(g), w\rangle$. Then $\phi(g) \geq 0$ for all $g \in P$.
    Since $P$ is finite, the image of $\phi$ is a finitely generated subgroup of $\mathbf{Q}$. In particular, the image is contained in a cyclic group, as required.
\end{proof}

A group $\Gamma$ is \textbf{locally indicable} if for every nontrivial finitely generated subgroup $\Lambda\le \Gamma$ there exists a surjective homomorphism $\phi: \Lambda\to \mathbf Z$.
Proposition \ref{prop:character} implies that Conradian-orderable groups are locally indicable as was established in \cite{article:Conrad}. The converse was first proven by \cite{article:brodski}, of which many proofs now exist \cite{book:Glass,article:RhemtullaRolfsen, article:Navas}.
\begin{thm}[\cite{article:Conrad, article:brodski}]\label{thm:locindicconrad}
    A group $\Gamma$ is locally indicable if and only if it is Conradian-orderable.
\end{thm}

\subsection{One-relator products}

In this section, we provide a quick summary of the theory of one-relator products, primarily following the work of Howie. We translate the results from the language of graphs of groups, in which most of the cited results are originally stated. 

In what follows, fix a generating set $S$ of a countable group $\Gamma$. A \textbf{word} in $S$ is a sequence $w = s_n\dots s_1$ of elements $s_i \in S^{\pm 1}$. A subword $v$ of $w$ is \textbf{proper} if it is not empty and not equal to $w$.

For a word $w$ in $S$, we abuse notation by identifying $w=s_n\dots s_1$ with the path $(1, s_1, s_{2}s_1, \dots, s_n\dots s_1)$ in the (left) Cayley graph.

We write $\overline w$ to denote the element of $\Gamma$ represented by $w$. A word is a \textbf{relation} of $\Gamma$ if $\overline w=1$. We denote by $\langle\!\langle \overline w\rangle\!\rangle\le \Gamma$ the normal closure of $\overline w$. 
A word is \textbf{embedded} if the associated path is injective (with the potential exception of endpoints). If the associated path starts and ends at the same group element, we say the word is an embedded relation. 
A word $w$ is a \textbf{proper power in $\Gamma$} if there exist a group element $g \in \Gamma$ and $m>1$ so that $\overline w=g^m$. Otherwise, we say $w$ is not a proper power (in $\Gamma$).

Given two countable groups $\Lambda, \Delta$ and a word $w$ in $\Lambda \cup \Delta$, we denote $\langle \Lambda, \Delta \mid w \rangle: = (\Lambda \ast \Delta)/\langle\!\langle \overline w \rangle\!\rangle$. We say it is a \textbf{one-relator product} if $\overline{w}$ is not conjugate in $\Lambda \ast \Delta$ to any element of $\Lambda$ or $\Delta$. This is equivalent to Howie's assumption regarding positive length of $w$, see the discussion at the top of page 447 in \cite{Howie:locInd}. 

We first mention a generalization of Magnus's Freiheitssatz \cite{Magnus} to one-relator products.
\begin{thm}[{\cite[Theorem 4.3]{Howie:Freiheitssatz}}]\label{thm:freih}
    Let $\Gamma = \langle \Lambda, \Delta \mid w \rangle$ be a one-relator product of locally indicable groups. Then, the canonical maps $\Lambda, \Delta \hookrightarrow \Gamma$ are injective.
\end{thm}

We require the following special case of a result of Howie. This was first proved in the context of one-relator groups by Weinbaum in \cite{article:weinbaum}.
\begin{prop}[Essentially in {\cite{Howie:locInd}}]\label{prop:HowieSpelling}
    Let $\Lambda=\langle S\rangle,\Delta=\langle T\rangle$ be locally indicable groups. Let $\Gamma = \langle \Lambda, \Delta \mid w \rangle$ be a one-relator product where $w=w_k\dots w_1$ is a word in $S\cup T$ and $w_1$ and $w_k$ are not both in $S$ or both in $T$. Assume that $w$ defines an embedded path in $\Cay(\Lambda\ast \Delta, S\cup T)$. Then $w$ is an embedded relation in $\Cay(\Gamma, S\cup T)$.
\end{prop}

Although it is very similar to  \cite[Corollary 3.4]{Howie:locInd}, Proposition \ref{prop:HowieSpelling} does not appear in this form in \cite{Howie:locInd} so we include a short derivation here.

\begin{proof}
    Because $w_1$ and $w_k$ come from different factors of the free product, they do not cancel and no word of the form $w_\ell\dots w_2w_1w_kw_{k-1}\dots w_j$ for $1 \leq \ell < j \leq k$ represents the identity in $\Lambda\ast \Delta$. Therefore, every cyclic permutation of $w$ 
    is embedded in $\Cay(\Lambda\ast \Delta, S\cup T)$, since subwords of a cyclic permutation are either subwords of $w$ or in the form above.

    We must verify that every proper subword $u$ of $w$ satisfies $\overline u\not=1$ in $\Gamma$. Any subword of $w$ is a terminal subword of a cyclic permutation of $w$, all of which are embedded by the above. Without loss of generality, we may thus assume $u=w_i\dots w_1$ is a proper terminal subword of $w$. Denote by $v=w_{k}\dots w_{i+1}$ the complementary subword so that $w=vu$. Assume towards a contradiction that $\overline{u} = 1$ in $\Gamma$ (and thus, $\overline{v} =1$ as well). Since $\overline u=1$ and $\overline v=1$, case ii) of \cite[Proposition 3.3]{Howie:locInd} must hold and one of $u$ or $v$ is the empty word, contradicting that $u$ is a proper subword. 
\end{proof}

The main result of \cite{Howie:locInd} establishes exactly the conditions under which a one-relator product of locally indicable groups remains locally indicable. 
\begin{thm}[{\cite[Theorem 4.2]{Howie:locInd}}]\label{thm:howielocindic}
    Let $\Gamma = \langle \Lambda, \Delta \mid w \rangle$ be a one-relator product of locally indicable groups $\Lambda, \Delta$. Then the following are equivalent:
    \begin{itemize}
        \item $\Gamma$ is locally indicable.
        \item $\Gamma$ is torsion-free.
        \item $w$ is not a proper power in $\Lambda \ast \Delta$.
    \end{itemize}
\end{thm}

\section{Positive systems and edge cutting}\label{sect:possyst}
In this section, we isolate positive systems, which will be the main tool in our approach. 
To motivate the definition, consider the case of an embedded relation $w$ in a right-ordered group $(\Gamma, <)$, where $\Gamma = \langle S \rangle$. For a given generator $s \in S$ appearing in $w$, assume that the $<$-minimum among the endpoints of edges labeled by $s$ appearing the path defined by $w$ is the identity, so that the other instances $P$ are all positive.

We will attempt to cut most $s$-edges in a Schreier graph of a suitable action, witnessing the connectivity of the endpoints using the copy of the relation that lives ``above'' it with respect to the order induced on the orbits. The main issue is that the path above usually contains other $s$-edges which may or may not be cut.  In a general right-ordered group, there is no reason to expect the Schreier graph to remain connected. 

However, let us assume for a sketch of proof that there is an order-preserving map $\phi:\Gamma \rightarrow \Z$ such that every element of $P$ has positive image. In this way, the map $\phi$ gives a way to quantify the informal idea of ``advancing'' through the order. One can then keep only the $s$-edges whose endpoints lie in a sparse but thick collection of cosets of $\ker \phi$, for example in (a free extension of) the coset process associated to $\ker \phi$. One can convince themselves that this leaves the graph connected. 

In this sketch, the order is not strictly necessary, only the map $\phi$ is important. 
In a Conradian-ordered group, there may not exist a single map $\phi$ such that all of the elements of $P$ have positive image. However, there always exist finite collections of maps with decreasing domain witnessing the positivity of the elements of $P$. We call such collections $P$-positive systems. In our argument, they will play the role of $\phi$ in that they will quantify ``advancement'' in the order.

\begin{definition}\label{def:posSystem}
    Let $\Gamma$ be a countable group and let $P \Subset \Gamma$ be a finite subset. A \textbf{$P$-positive system (on $\Gamma$)} is a sequence $\{P_i, \phi_i\}_{1\le i \leq L}$ such that:
    \begin{enumerate}
        \item $P_1 = P$ and $P_{i+1} = P \cap \ker\phi_i$ for $1\le i<L$,
        \item $P\cap \ker \phi_L = \emptyset$,
        \item $\phi_i: \langle P_i \rangle \rightarrow \Z$ is a surjective group homomorphism with $\phi_i(g) \geq 0$ for each $g \in P_i$ and each $1\le i \leq L$.
    \end{enumerate}
    The \textbf{width} of the system is the maximum value of $\phi_i(g)$ taken over $g\in P_i$ and $1\le i\le L$.
\end{definition}

For a given $P$, there may exist positive systems even if the group itself is not Conradian-orderable, such as when it admits a Conradian-orderable quotient.
\begin{prop}\label{prop:quotientLocIndPositive}
    Let $\Gamma$ be a countable group with a group homomorphism $\pi:\Gamma\to \Gamma'$ to a Conradian-ordered group $(\Gamma', <)$. Let $P\Subset \Gamma$ be a finite set such that every element of $\pi(P)$ is positive. 
    Then there exists a $P$-positive system on $\Gamma$.
\end{prop}

\begin{proof}
    We construct a $P$-positive system recursively.
    Set $P_1=P$. 
    
    Suppose $P_i\not=\emptyset$ is already defined. We define $\phi_i: \langle P_i \rangle \to \mathbf Z$. The order $<$ is Conradian on $\langle \pi(P_i)\rangle\subseteq \Gamma'$ and since $P_i$ is finite, by Corollary \ref{cor:fgsubgroupshavecharacters} there exists a surjective homomorphism $\chi_i: \langle \pi(P_i)\rangle\to \mathbf Z$ with $\chi_i(\pi(g))\ge 0$ for $g\in P_i$. Set $\phi_i:=\chi_i\circ \pi$.

    If $P\cap \ker \phi_i=\emptyset$ the process terminates. Otherwise, set $P_{i+1}=P\cap \ker \phi_i\subset P_i$ and note that $P_{i+1}$ is strictly smaller than $P_i$ since there must exist a $p\in P_i$ with $\chi_i(\pi(p))>0$. Thus the process terminates at some stage $L\le|P|$ and the system $\{P_i, \phi_i\}_{1\le i \leq L}$ is a $P$-positive system on $\Gamma$.
\end{proof}

\begin{lemma}\label{lem:disamb-induction}
    Let $\Gamma\acts X$ be an action of a countable group $\Gamma$. Let $P \Subset \Gamma$ be a finite subset, and let $\{P_i, \phi_i \}_{1\le i \leq L}$ be a $P$-positive system of width $M$. Assume that for some $n \in \N$ and for all $1 \leq i \leq L$, there are $P_i$-equivariant maps $\psi_{i,n}: X \rightarrow \Z/n\Z$, i.e., $\psi_{i,n}(g.x)=\psi_{i,n}(x)+\phi_i(g)$ for $g\in P_i$.
    
    Then, for $x \in X$ and for every sequence $\{g_j\}_{1\le j\leq n^L}$ of elements in $P$, there exist $1\le j \leq n^L$ and $1\le i\leq L$ such that 
    \[\psi_{i,n}(g_j\cdots g_1.x) \in \{0,1,2,...,M-1\}.\]
\end{lemma}

\begin{proof}
Define $J_i=\{1\le j\le n^L\mid g_j \in P_i- P_{i+1}\}$ where we take $P_{L+1}=\emptyset$. 
By the pigeonhole principle, in any sequence of length $n^L$, there must exist $1\le i\le L$ and an interval $1\le a\le b\le n^L$ such that $|J_i\cap \{a,a+1,\dots, b\}|\ge n$ and $g_j\in P_i$ for every $a\le j\le b$. Indeed, if no such $i$ and interval exist, then $|J_1| \leq n-1$. The elements of $J_1$ cut the sequence into at most $n$ intervals. Inside each interval there must be at most $n-1$ elements of $J_2$, cutting each interval into at most $n$ subintervals. Continuing in this fashion the total length of the sequence must be at most $(n-1)+(n-1)n+(n-1)n^2+\cdots +(n-1)n^{L-1}=(n-1)(1+n+\cdots+n^{L-1})=n^L-1$.

Fix $1\le i\le L$ and interval $1\le a<b\le n^L$ as above, so that $g_j\in P_i$ and $0\le \phi_{i}(g_j)\le M$ for all $a\le j\le b$, and there exist at least $n$ many $a\le j\le b$ so that $\phi_{i}(g_j)$ is nontrivial.  
Fix $x\in X$ and let $y_j=g_j\cdots g_1.x$ for $1\le j\le n^L$ with $y_0=x$. Since $\psi_{i,n}(y_j)=\psi_{i,n}(y_{j-1})+\phi_i(g_j)$, we have that the sequence $(\psi_{i,n}(y_j))_{a\le j\le b}$ must intersect $\{0,\dots, M-1\}$, as required.
\end{proof}

If $p=(x_1, x_2, \dots, x_n)$ is a path in a graph, we say an oriented edge $(y,z)$ is \textbf{contained} in the path $p$, denoted $(y,z)\in p$, if there exists a $1\le i\le n-1$ so that $x_i=y$ and $x_{i+1}=z$. For an edge $(y,z)$ in a graph, we write $(y,z)^{-1}:=(z,y)$. We extend this notation to paths, for example $(x,y,z)^{-1} = (z,y,x)$.

\begin{thm}\label{thm:cutRelator}
    Let $\Gamma = \langle S \rangle$ be a countable group, $w^m$ be an embedded relation and $s \in S$ be a generator that appears in $w^m$. Let $Q = \{g \in \Gamma : (g, sg) \text{ or } (sg,g) \in w^m\}$. Suppose that $\{1, \overline w, ..., \overline w^{m-1}\}\subseteq Q$ and let $P=Q-\{1, \overline w, ..., \overline w^{m-1}\}$.
    If $P$ is not empty, assume that there exists a $P$-positive system. 
    
    Then there exists a free pmp action $\Gamma\acts (X,\mu)$ such that for all $\varepsilon > 0$, there is a Borel set $A_\varepsilon\subseteq X$ with $\mu(A_\varepsilon)\ge \frac 1m -\varepsilon$ such that 
    $\mathcal{G}_\varepsilon := \Sch(\Gamma \acts X, S) - \{(s.x,x)^{\pm1}: x \in A_\varepsilon\}$
    is a graphing of $\Gamma\acts (X,\mu)$. 
\end{thm}

\begin{proof}
    If $P$ is empty, then take $\Gamma\acts (X,\mu)$ to be any free pmp action and let $A_{\varepsilon}$ be a Borel transversal for the restricted action $\langle \overline{w}\rangle \acts X$. Then $\mu(A_{\varepsilon})=\frac 1m$ since the action is free and $\overline{w}$ has order $m$. The graph $\mathcal{G}_\varepsilon$ as in the statement remains connected as we removed one edge from each cycle given by $w^m$. Indeed, since $P$ is empty, any pair of $s$ edges in the same $w^m$ cycle have origins in the same $\overline w$ orbit, while $A_\varepsilon$ is a transversal for this action.
    
    Assume henceforth $P$ is not empty and let $\{P_i, \phi_i\}_{i \leq L}$ be a $P$-positive system. For $n \in \N$, the group $\Delta_{i,n} := \phi_i^{-1}(n\Z)$ has finite index in $\langle P_i\rangle$. Let $\Gamma \acts (X,\mu)$ be a free extension of the $\Gamma$-coset process associated to \[\{\Delta_{i,n}:{1\le i\le L} \text{ and } n\in \N\}.\]
    Recall from Proposition \ref{prop:finIndexCosetProcess} and Observation \ref{obs:cosetProcessFactors} that on $\Gamma$-invariant conull sets $X_{i,n}\subseteq X$, there exist $P_i$-equivariant Borel maps $\psi_{i,n}: X_{i,n} \rightarrow \langle P_i \rangle/\Delta_{i,n} \cong \Z/n\Z$. The set $X':=\bigcap_{i,n}X_{i,n}$ remains $\Gamma$-invariant and conull. Thus the action $\Gamma \acts X'$ and maps $\psi_{i,n}$ satisfy the hypotheses of Lemma \ref{lem:disamb-induction}. 
    
    Observe that $\overline w$ has order $m$. Indeed, $\overline w^m=1$ and for $1\le k < m$, $\overline w^k\not=1$ since $w^m$ is embedded by hypothesis. Let $C\subseteq X'$ be a Borel transversal for the restricted action $\langle \overline{w}\rangle \acts X$. Then $\mu(C)=\frac 1m$ since the action is free and $\overline{w}$ has order $m$.

    Fix $\varepsilon >0$ and let $N=\lceil LM\varepsilon^{-1}\rceil$, where $M$ is the width of $\{P_i, \phi_i\}_{i\le L}$. Define 
    \[
    A_\varepsilon=C-\bigcup_{1\le i\le L} \psi_{i,N}^{-1}(\{0,\dots, M-1\})
    \]
    and observe $\mu(A_\varepsilon)\ge\mu(C)-\frac{LM}{N}\ge \frac{1}{m} -\varepsilon$.
    
    Since $1\in Q$, after possibly inverting $w^m$, we may assume that $w^m= w's$. After removing $(1,s)$ from the cycle labeled by $w^m$, the remaining edges form a path $\rho$ from $s$ back to $1$ that is labeled by $w'$.

    For a path $p = (g, s_0g, s_1s_0g,..., s_k...s_0g)$ in $\Cay(\Gamma, S)$ and an $x\in X$, define $\hat p^x$ to be the path in $\Sch(\Gamma\acts X,S)$ given by
    \[
    \hat p^x=(g.x,s_0g.x,s_1s_0g.x,\dots, s_k\dots s_0g.x).
    \]
    Given $x \in A_\varepsilon$, we wish to witness the fact that the endpoints of a removed edge $e_x = (s.x,x)$ are still connected in $\mathcal{G}_\varepsilon$. Towards that end, define a ``replacement operator'' 
    \[
    \theta(e_x) = \hat \rho^x.
    \]
    Since $\rho$ is a path from $s$ to $1$,  $\theta(e_x)$ is a path from $s.x$ to $x$ and so has the same endpoints as $e_x$. For the reverse orientation, define $\theta (e_x^{-1})=(\hat\rho^x)^{-1}$.
    If $e\in \G_\varepsilon$ and thus not of the form $e_x^{\pm 1}$ for $x\in A_\varepsilon$, then define $\theta(e)=e$. Extend $\theta$ to all paths in $\Sch(\Gamma\acts X,S)$ by concatenation. For $r \geq 0$, denote $\theta^r$ to be the $r$-fold application of $\theta$, where $\theta^0 = \id$.
    
    Suppose that an edge $e_y^\delta$ occurs in $\theta(e_x)$ where $x,y\in A_\varepsilon$ and $\delta\in \{\pm 1\}$. Observe in this case that $y=g.x$ for some $g\in P$. Indeed, it is clear that $g\in Q-\{1\}$, since $e_y^\delta \in \hat\rho^x$ and thus $(sg, g)^\delta$ is contained in $\rho \subset w^m$. Note that $\overline{w}^k.A_\varepsilon \cap A_\varepsilon = \emptyset$ for $1\le k\le m-1$ since $C$ is a fundamental domain for $\langle\overline w \rangle$ and $A_\varepsilon \subset C$. Therefore, $g\not\in\langle \overline w\rangle$ and so $g\in P$. 
    
    \begin{claim}\label{claim:mainthm}
    Let $x\in A_\varepsilon$ and $r\ge 0$. Suppose that there exist a $y\in A_\varepsilon$ and $\delta\in \{\pm 1\}$ such that $e_y^\delta\in \theta^r(e_x)$.  Then there exist $g_1, \dots, g_{r} \in P$ such that $y = g_{r}\dots g_1.x$ and $g_{j}\dots g_1.x \in  A_\varepsilon$ for all $0\le j \leq r$.
    \end{claim}

    \begin{proof}[Proof of claim]
        We argue by induction on $r$. For the base case $r = 0$, $\theta^0(e_x)=e_x$ so the claim holds with the empty sequence. 
        
        Assume the claim holds for $r\ge 0$ and consider $e \in \theta^{r+1}(e_x)$ with $e=e_y^\delta$ for some $y\in A_\varepsilon$ and $\delta\in\{\pm 1\}$. 
        Then there is an edge $e'\in \theta^r(e_x)$ with $e\in\theta(e')$. 
        The edge $e'$ must also be of the form $e_{y'}^{\delta'}$ for some $y'\in A_\varepsilon$ and $\delta'\in\{\pm 1\}$ as otherwise $\theta(e')=e'$. 
        By the observation above, there exists a $g_{r+1}\in P$ so that $g_{r+1}.y'=y$.

        By the induction hypothesis, there exist $g_1, \dots, g_{r}\in P$ satisfying $y' = g_{r}\dots g_1.x$ and $g_{j}\dots g_1.x \in A_\varepsilon$ for all $0\le j \leq r$. The sequence $g_1,\dots, g_{r+1}$ now satisfies the claim. 
    \end{proof}
    By Lemma \ref{lem:disamb-induction} and Claim \ref{claim:mainthm}, it must be that $\theta^{N^L}(e_x) \subset \mathcal{G}_\varepsilon$. Therefore, $s.x$ and $x$ remain connected in $\mathcal G_\varepsilon$  for all $x\in X'$, and thus on a conull set the connected components are the same as those of $\Sch(\Gamma\acts X,S)$ finishing the proof.
\end{proof}

\begin{remark}\label{rmk:nofixedprice}
    A standard strategy for proving fixed price involves computing cost for an action which is weakly contained in every free pmp action, such as the Bernoulli shift \cite{AbertWeiss}, and using monotonicity of cost under weak containment \cite[Corollary 10.14]{book:GlobalAspects} to establish a global upper bound. 
    
    Our techniques do not compute an upper bound on the cost of the Bernoulli shift. 
    Indeed, the action $\alpha$ constructed in Theorem \ref{thm:cutRelator} is not weakly contained in the Bernoulli shift $\beta$ of $\Gamma$ if $\Lambda=\langle P\rangle$ is nonamenable. 
    Suppose toward a contradiction that $\alpha$ is weakly contained in $\beta$. 
    The construction gives that $\alpha|_\Lambda$ has a nontrivial finite factor $\Lambda\acts \Lambda/\Delta$ where $\Delta=\phi_1^{-1}(n\mathbf Z)$ for any $n\ge2$. In particular, the action $\alpha|_\Delta$ is not ergodic.
    Since $\Lambda$ is nonamenable, $\Delta$ is nonamenable and thus the Bernoulli shift $\beta|_\Delta$ is strongly ergodic. This contradicts that $\alpha|_\Delta$ is weakly contained in $\beta|_\Delta$.
\end{remark}

\section{Applications to cost of groups}

In this section, we will apply the cycle-cutting mechanism developed in Section \ref{sect:possyst} to study one-relator products. We will make use of a special case of a result of Peterson and Thom which we record here for convenience of the reader.

\begin{thm}[Special case of {\cite[Theorem 3.2]{PetersonThom}}]\label{thm:pt}
Let $\Gamma$ be an infinite countable discrete group and assume that there exist infinite subgroups $\Lambda$ and $\Delta$ such that $\Gamma = \langle \Lambda, \Delta\mid w^{m}\rangle$ for some $w\in \Lambda*\Delta$ and $m\ge 1$. We assume that the presentation is irredundant in the sense that $\overline w^k \neq 1 \in \Gamma$ for $1 < k < m$. Then, the following inequality holds:
\[\beta_1^{(2)}(\Gamma) \geq \beta_1^{(2)}(\Lambda) +\beta_1^{(2)}(\Delta) +1- \frac{1}{m}.\]
\end{thm}

The novelty in the following theorem is the case when $\Gamma$ is a one-relator product. For convenience, we also include in the statement the case where $\Gamma\cong \Lambda\ast \mathbf Z/m\mathbf Z$. 

\begin{thm}\label{thm:oneRelatorProduct}
    Let $\Gamma=\langle \Lambda, \mathbf Z \mid w^m\rangle$ where $\Lambda$ is locally indicable and $m\ge 1$. Assume $w$ is not a proper power and that $\overline{w}$ is not conjugate to an element of $\Lambda$. 
    Then 
    \[
    \cost(\Gamma)\le \cost(\Lambda)+1-\frac{1}{m}.
    \]

    Moreover, if $\beta_1^{(2)}(\Lambda)=\cost(\Lambda)-1$, then $\beta_1^{(2)}(\Gamma)=\cost(\Gamma)-1=\cost(\Lambda)-\frac{1}{m}$.
\end{thm}

\begin{proof}[Proof of Theorem \ref{thm:oneRelatorProduct}]
    Let $\Z=\langle s\rangle$. Observe that if $\overline{w}$ is conjugate to an element of $\langle s\rangle$, then $\overline{w}$ is in fact conjugate to $s^{\pm 1}$, since it is not a proper power. If $m=1$, then $\Gamma\cong \Lambda$ and the statement is obvious. Otherwise $m\ge 2$, $\Gamma\cong \Lambda\ast \mathbf Z/m\mathbf Z$, and so  $\cost (\Gamma)\le\cost (\Lambda)+1-\frac 1m$ (in fact, this is an equality by \cite[Th\'eor\`eme VI.7'']{Ga00}). In this case, the moreover follows from the formula for $\ell^2$-Betti numbers of free products, but also from the argument presented further in this proof. So we may assume $\overline{w}$ is not conjugate to a power of $s$, i.e. that $\Gamma$ is a one-relator product. 

   Theorem \ref{thm:howielocindic} states that $\Gamma'=\langle\Lambda, \mathbf Z\mid \overline w\rangle$ is locally indicable.
    Let $\pi:\Gamma\to \Gamma'$ be the quotient map and fix a Conradian order $<$ on $\Gamma'$, by Theorem \ref{thm:locindicconrad}.
    
    The element $\overline w \in \Lambda \ast \Z$ has a normal form $\overline w=s^{n_k}\lambda_k \dots s^{n_2}\lambda_2 s^{n_1}\lambda_1$ with $n_i\not=0$ for $1\le i\le k-1$ and $\lambda_i\not=1$ for $2\le i\le k$. We may thus conjugate to obtain a word $w'=s^{m_\ell}\lambda_\ell' \dots s^{m_2}\lambda_2's^{m_1}\lambda_1'$ with $m_i\not=0$ and $\lambda_i'\not=1$ for $1\le i\le \ell$, since $w$ is not conjugate to an element of $\Lambda$ or $\langle s\rangle$. Thus without loss of generality, we may assume $w=w'$ since $\Gamma =\langle\Lambda, \mathbf Z\mid(w')^m\rangle$ and $\Gamma'=\langle\Lambda, \mathbf Z\mid w'\rangle$. 
    
    Because of the free product structure, both $w$ and $w^m$ are embedded paths in $\Cay(\Lambda\ast \Z, \Lambda\cup\{s\})$.
    By Proposition \ref{prop:HowieSpelling}, $w$ is an embedded relation in $\Cay(\Gamma', \Lambda\cup \{s\})$ and $w^m$ is an embedded relation in $\Cay(\Gamma, \Lambda\cup \{s\})$.

    Let $Q=\{g\in \Gamma: (g,sg)\text{ or }(sg,g)\in w^m\}$ and choose $q\in Q$ so that $\pi(q)$ is $<$-minimal in $\pi(Q)$. By cyclically permuting $w^m$ and $w$ to get $v^m$ and $v$, we obtain $Q_0=\{g\in \Gamma: (g,sg)\text{ or }(sg,g)\in v^m\}=Qq^{-1}$, and hence $1\in Q_0$. If $\pi(hq^{-1})\not=1$ for some $h\in Q$, then $\pi(q)<\pi(h)$ and so $1<\pi(hq^{-1})$. Since cyclic permutation does not alter the fact that $v$ and $v^m$ are embedded relations, we may assume without loss of generality that $1$ is in $Q$ and is $<$-minimal among the set $\pi(Q)$.
    Observe that in this case $\{1,\overline w,\dots, \overline {w}^{m-1}\}\subseteq Q$ and let $P=Q-\{1,\overline w,\dots,\overline w^{m-1}\}$. 
    
    We claim that $\pi(P)$ consists of positive elements of $\Gamma'$. Indeed, if $h\in P$, it may be written as $h=\overline{w_h}\,\overline w^r$ for some non-empty initial subword $w_h$ of $w$. 
    Since $w$ is embedded in $\Cay(\Gamma', \Lambda\cup \{s\})$, we have $\pi(h)=\pi(\overline {w_h}\,\overline w^r)=\pi(\overline {w_h})\not=1$. By our choice of cyclic permutation, $1$ is minimal among $\pi(Q)$ and so $1<\pi(h)$. 

    By Proposition \ref{prop:quotientLocIndPositive} there exists a $P$-positive system and thus the hypotheses of Theorem \ref{thm:cutRelator} are satisfied.

    Fix $\varepsilon>0$ and apply Theorem \ref{thm:cutRelator} with respect to the generator $s$ and the relation $w^m$ to obtain a free pmp action $\Gamma\acts (X,\mu)$ along with a measurable set $A \subset X$ with $\mu(A)\ge \frac{1}{m}-\frac{\varepsilon}{2}$ such that the Borel graph $\mathcal{G}_X = \Sch(\Gamma \acts X, \Lambda\cup \{s\}) - \{(s.x,x)^{\pm 1} : x\in A \}$ is a graphing of the action $\Gamma\acts X$.

    Fix a free pmp action $\Lambda\acts (Y,\nu)$ along with a graphing ${\mathcal H}_Y$ of the orbit equivalence relation such that $\mathcal C_\nu(\mathcal H_Y)\le \cost(\Lambda)+\frac{\varepsilon}{2}$.  By the Freiheitssatz (Theorem \ref{thm:freih}), the natural map $\Lambda\hookrightarrow\Gamma$ is injective. 
    We thus regard $\Lambda$ as a subgroup of $\Gamma$.
    Then, the \textit{coinduction} construction applied to $\Lambda \acts (Y,\nu)$ gives a free pmp action $\Gamma \acts (Z, \eta)$ and a $\Lambda$-equivariant map $\iota: Z \rightarrow Y$ with $\iota_*\eta=\nu$. See \cite[Section II.10.G]{book:GlobalAspects} for the construction and basic properties.

    Consider the diagonal action $\Gamma\acts (X\times Z, \mu\times\eta)$. Define a graphing $\mathcal H$ of the action $\Lambda\acts X\times Z$ of the same cost as $\mathcal H_Y$ by
    \[
        \mathcal H :=\left\{\big((x,z),\lambda.(x,z)\big )\;\middle|\; \big(\iota(z), \lambda.\iota(z)\big)\in \mathcal H_Y\right\}.
    \]
    Extend this to a Borel graph $\G:=\mathcal H\cup \left\{\left((x,z),s.(x,z)\right)^{\pm1}\;\middle|\; x\in X-A\right\}$ of cost 
    \begin{equation}\label{costComputation}
    \mathcal C_{\mu\times \eta} (\G)\le \mathcal C_{\mu\times \eta} (\mathcal H)+\mu(X-A)\le \cost(\Lambda)+1-\tfrac{1}{m}+\varepsilon.
    \end{equation}  
    
    We claim that $\mathcal G$ is a graphing of $\Gamma\acts X\times Z$. 
    We must verify that $g.(x,z)$ is connected to $(x,z)$ by a path in $\G$ for almost every $(x,z)\in X\times Z$ and every $g\in \Gamma$. The points $g.x$ and $x$ are almost surely connected in $\G_X$ by a path of edges either labeled by elements of $\Lambda$ or labeled by $s$ with a basepoint in $X-A$. Lift this path to a path $p$ in $X\times Z$ starting at $(x,z)$ and ending at $g.(x,z)$ by following the same sequence of labels. Every $s$-labeled edge in $p$ is in the graph $\G$ by definition and the endpoints of every $\Lambda$-labeled edge in $p$ are connected in $\mathcal H\subseteq \G$.

    Since $\G$ is a graphing of a free pmp action of $\Gamma$, $\cost(\Gamma) \leq \mathcal C_{\mu \times \eta}(\G)$, thus sending $\varepsilon$ to $0$ in Equation~\eqref{costComputation} concludes the proof. 

    Regarding the moreover, since the relation $w^m$ is embedded in $\Gamma$, the presentation is irredundant. We thus apply Theorem \ref{thm:pt} to obtain $\beta_1^{(2)}(\Gamma)\ge \beta_1^{(2)}(\Lambda)+1-\frac{1}{m}$. Therefore, if $\beta_1^{(2)}(\Lambda)=\cost(\Lambda)-1$, then by the inequality in \cite[Corollaire 3.23]{Ga02L2},
    \[
    \cost(\Lambda)-\frac{1}{m}=\beta_1^{(2)}(\Lambda)+1-\frac{1}{m}\le \beta_1^{(2)}(\Gamma)\le \cost(\Gamma)-1\le \cost(\Lambda)-\frac{1}{m}.\qedhere
    \]
\end{proof}

The main theorem on one-relator groups is a special case of Theorem \ref{thm:oneRelatorProduct} where $\Lambda$ is a free group, using the fact that the cost of a free group $\mathbf F_k$ on $k$ generators is $k = \beta_1^{(2)}(\mathbf{F}_k) + 1$ \cite{Ga00}.
In fact, Theorem \ref{thm:oneRelatorProduct} implies Question \ref{quest:costvsl2} has a positive answer for a larger class of groups.

From \cite{Howie:locInd}, say a presentation $\langle S\mid R\rangle$ is \textbf{reducible} if every finite subpresentation is of the form $\langle s_1,\dots, s_k, t\mid w_1,\dots, w_\ell, u\rangle$ where each of the $w_i$ is a word in $s_1,\dots, s_k$ and $\overline u$ is not conjugate in $\Gamma \ast \langle t\rangle$ to an element of $\Gamma=\langle s_1,\dots, s_k\mid w_1,\dots, w_\ell\rangle$. See \cite{article:lintonjaikin} for more about reducible presentations.

Given a presentation $\langle S\mid R\rangle$ and relation $w\in R$, let $\Gamma_w=\langle S\mid R-\{w\}\rangle$. We say the presentation \textbf{has no proper powers} if $w$ is not a proper power in $\Gamma_w$ for every $w$. If $\langle S\mid R\rangle$ has no proper powers, then so does every subpresentation.

\begin{cor}\label{cor:reducible}
    Let $\Gamma=\langle S \mid R\rangle$ be an infinite group with a finite reducible presentation with no proper powers. Then \[\cost(\Gamma)-1=\beta^{(2)}_1(\Gamma)=|S|-|R|-1.\] 
\end{cor}

\begin{proof}
    We induct on $|R|$. If $|R| = 0$, the group is free and the statement holds by \cite{Ga00}.
    
    Else, by reducibility, there exist a generator $s \in S$ and a relation $w \in R$ so that every element of $R' := R- \{w\}$ is a word in $S' := S - \{s\}$ and $\overline{w}$ is not conjugate in $\Gamma' \ast \langle s \rangle$ to an element of $\Gamma'$, where $\Gamma' = \langle S' \mid R' \rangle$. Further, since $\langle S \mid R\rangle$ has no proper powers,  $\langle S' \mid R' \rangle$ has no proper powers, and so $w$ is not a proper power in $\Gamma_w = \Gamma' \ast \langle s \rangle$. Any subpresentation of a reducible presentation is clearly reducible. 
    As a result $\Gamma\cong\Gamma'\ast \langle s\rangle/\langle\!\langle \overline w\rangle\!\rangle$, where $\Gamma'$ satisfies the assumption of the Corollary, with $|S'| = |S| - 1$ and $|R'| = |R| - 1$.
    
    By the induction hypothesis, $\cost(\Gamma')-1=\beta^{(2)}_1(\Gamma')=|S'|-|R'|-1$. Since $\Gamma'$ has a reducible presentation with no proper powers, $\Gamma'$ is locally indicable \cite[Corollary 4.5]{Howie:locInd} and we may apply Theorem \ref{thm:oneRelatorProduct} (with $m=1$) to obtain 
    \[\cost(\Gamma)-1=\beta^{(2)}_1(\Gamma)=|S'|-|R'|-1 = |S| - |R| - 1.\qedhere\]
\end{proof}

\begin{remark}
    Since Theorem \ref{thm:oneRelatorProduct} also applies to torsion relators, Corollary \ref{cor:reducible} can be extended to reducible presentations where the only torsion relation $w^m$ is the ``top'' one to obtain $\cost(\Gamma)-1=\beta_1^{(2)}(\Gamma)=|S|-|R|-\frac{1}{m}$.
\end{remark}
\bibliography{references}
\bibliographystyle{alphaurl}

\end{document}